\documentclass[final]{siamltex}

\usepackage{epsfig}
\usepackage{amsmath,amsfonts,amssymb}
\usepackage{algpseudocode,algorithm}
\usepackage{url}

\newcommand{\knorm}[2]{\vert\vert #1 \vert\vert_{#2}}

\newcommand{\Cset}{\ensuremath{\mathbb{C}}}

\newcommand{\Oh}[1]{\ensuremath{O(#1)}}

\newcommand{\matlab}{{\rm Matlab}}
 
\newcommand{\ilut}{{\rm ILUT}}
\newcommand{\ainv}{{\rm AINV}}  
\newcommand{\bicgstab}{{\rm BiCGSTAB}}

\newcommand{\absnorm}[1]{\vert #1 \vert}

\newcommand{\nz}[1]{\ensuremath{\mbox{\rm nz}(#1)}}
\newcommand{\Wcal}{\ensuremath{\mathcal{W}}}
\newcommand{\Vcal}{\ensuremath{\mathcal{V}}}

\newcommand{\mcSF}{{\tt MC64}}
\newcommand{\paif}{{\rm PAIF}} 
\newcommand{\diaf}{{\rm DIAF}}

\def\setC{\,\mathbb{C}}

\makeatletter
\newcounter{example}
 \renewcommand{\theexample}{\thesection.\arabic{example}}
 \newenvironment{example}
 {\refstepcounter{example}\paragraph {\it Example \theexample}}
 {\vspace{1ex}}
\@addtoreset{example}{section}
\makeatother

\title{PRECONDITIONING WITH
DIRECT APPROXIMATE FACTORING OF THE INVERSE}
\author{
Mikko Byckling\thanks{
CSC - IT Center for Science,
P.O. Box 405,
02101 Espoo, 
Finland
({\tt Mikko.Byckling@csc.fi}).}
\and 
Marko Huhtanen\thanks{
Mathematics Division, Department of Electrical and Information Engineering,
University of Oulu, 
P.O. Box 4500,
FIN-90401 Oulu, 
Finland,
({\tt Marko.Huhtanen@oulu.fi}). }
}

\begin{document}
\maketitle
\begin{abstract}
To precondition a large and sparse
linear system, two direct methods for  
approximate factoring of the inverse are devised. 
The algorithms are fully parallelizable and appear to be 
more robust than the iterative methods suggested for the task.
A method to compute one of 
the matrix subspaces optimally is derived.
Possessing a considerable amount of flexibility, 
these approaches extend the approximate inverse preconditioning
techniques in several natural ways. 
Numerical experiments are given to illustrate the performance
of the preconditioners on a number of challenging  
benchmark 
linear systems.
\end{abstract}


\pagestyle{myheadings}
\thispagestyle{plain}

\markboth{M. BYCKLING  AND M. HUHTANEN 
}{DIRECT APPROXIMATE FACTORING} 

\section{Introduction}  
Approximate factoring of the inverse means parallelizable 
algebraic techniques for preconditioning 
a linear system involving
a large and sparse nonsingular matrix $A\in \setC^{n \times n}$.
The  idea is to multiply 
$A$ by a matrix $W$ from the right (or left) with the aim at having a matrix $AW$ which
can be approximated with an easily invertible matrix.\footnote{Direct methods
are typically devised in this way, i.e., both the  LU and QR factorization
can be interpreted such that  the purpose is to multiply
$A$ with a matrix 
from the left so as to have an upper triangular, i.e., an easily invertible 
matrix.}
As opposed to  the usual paradigm of preconditioning, iterations 
are not expected to converge rapidly for $AW.$  
Instead, the task can be interpreted as that of solving
\begin{equation}\label{aivbh}
\inf_{W\in \mathcal{W},\, V\in \mathcal{V}} \left|\left| AWV^{-1}-I \right|\right|_F
\end{equation}
approximately 
by linearizing the problem appropriately \cite{HR,BHU}.
Here $\mathcal{W}$ and $\mathcal{V}$ 
are  
nonsingular sparse standard 
matrix subspaces of $\setC^{n \times n}$ with the property that
that the  nonsingular elements of $\mathcal{V}$ 
are assumed to allow
a rapid application of the inverse.
Approximate solutions to this problem can be generated with the power method
as suggested in \cite{BHU}.
In this paper, direct methods are devised for approximate
factoring based on solving
\begin{equation}\label{aivan}
\min_{W\in \mathcal{W},\; V\in \mathcal{V}} 
\left|\left| AW-V \right|\right|_F
\end{equation}
when the columns of either $W$ or $V$ being constrained to 
be of fixed norm.
These two approaches allow, once  the matrix subspace
$\mathcal{W}$ has been fixed, 
choosing 
the matrix subspace $\mathcal{V}$ in an optimal way.

The first algorithm solves
\eqref{aivan} when 
the columns of $V$ are constrained to be of fixed norm. 
Then the matrix subspaces $A\mathcal{W}$ and $\mathcal{V}$
are compared as such while other properties of $A$ are
largely overlooked.
The second algorithm solves the problem when 
the columns of $W$ are constrained to be of fixed norm,  
allowing taking properties of $A$  more into account.  
In \cite{BHU} the approach to this end was 
based on approximating the smallest
singular value of the map 
\begin{equation}\label{singva}
W\longmapsto (I-P_{\mathcal{V}})AW
\end{equation}
from $\mathcal{W}$ to $\setC^{n \times n}$ with the power iteration.
Here $P_{\mathcal{V}}$ denotes the orthogonal projector on $\setC^ {n \times n}$
onto $\mathcal{V}$. 
The second algorithm devised in this paper is a direct method 
for solving the task.

The algorithms proposed extend 
the standard approximate inverse computational techniques
in several ways.
(For sparse approximate inverse computations, see \cite[Section 5]{BE},  
\cite{GH} and \cite[Chapter 10.5]{SA} 
and references therein.) 
Aside from possessing an abundance of degrees of freedom, we have 
an increased amount of
optimality if we  suppose the matrix subspace  $\mathcal{W}$ to be given.
Then computable conditions can be formulated for optimally choosing 
the matrix subspace $\mathcal{V}$. This is achieved without any significant
increase in the computational cost. 
In particular, only a columnwise access 
to the entries of $A$ is required.\footnote{Accessing 
the entries of the adjoint can be costly
in parallel computations.}

We aim at maximal parallelizability by solving 
the minimization problem \eqref{aivan}
columnwise. The cost of such a high parallelism is the need to have
a mechanism to  somehow control the conditioning
of the factors. After all, parallelism means 
performing  computations locally and independently.
Also this can be achieved without any significant
increase in the computational cost.

Although the choice of the matrix subspace $\mathcal{W}$ is 
apparently less straightforward, some ideas are suggested to this end. Here 
we cannot claim achieving optimality, except that once done,
thereafter $\mathcal{V}$ can be generated in an optimal way. In particular,
because there are so many alternatives to generate matrix
subspaces, many ideas outlined in this paper are certainly not
fully developed and need to be investigated more
thoroughly.

The paper is organized as follows. In Section \ref{method} two algorithms
are devised for approximate factoring of the inverse. 
Section \ref{sec3} is concerned with ways to choose the matrix
subspace $\mathcal{V}$ optimally. Related stabilization schemes are suggested.
In Section \ref{Sec4} heuristic Al schemes are suggested for
constructing the matrix subspace $\mathcal{W}$.
In Section 5 numerical experiments are conducted.
The toughest benchmark problems from \cite{Benzi2000}
are used in the tests.

\section{Direct approximate factoring of the inverse}\label{method}  
\label{sec:DIAF}

In what follows, two algorithms are devised for computing matrices $W$ and $V$ to have 
an approximate factorization 
\begin{equation}\label{afac}
A^{-1}\approx WV^{-1}
\end{equation}
 of the inverse of a given sparse
nonsingular  matrix $A\in \setC^{n \times n}$. 
The factors $W$ and $V$
are assumed to belong to given sparse standard matrix subspaces $\mathcal{W}$ 
and $\mathcal{V}$ of 
$\setC^{n \times n}$. A matrix subspace
is said to be standard 
if it has a basis consisting of standard
basis matrices.\footnote{Analogously to the standard basis vectors of $\setC^n$,
a standard basis matrix of $\setC^{n \times n}$ has exactly one
entry equaling one while its other entries are zeros.}
This allows maximal parallelizability by the fact that then
the arising computational problems can be 
solved columnwise independently. 
Of course,
parallelizability is imperative to fully exploit the processing power 
of modern computing architectures. 

\subsection{First basic algorithm}
\label{sec:qralg}
Consider the minimization problem \eqref{aivan} 
under the assumption that 
the columns of $V$ are constrained to be unit vectors,
i.e., of norm one. 
Based on the sparsity structure of $\mathcal{W}$ and 
the corresponding columns of $A$,
the aim is at first choosing $V$ optimally. Thereafter $W$ is
determined optimally. 

To describe the method, denote by $w_j$ and $v_j$ the $j$th
columns of $W$ and $V$. The column $v_j$ is computed first as follows.
Assume there can appear $k_j\ll n$ nonzero
entries in $w_j$ at prescribed positions
and denote by $A_j\in \setC^{n\times k_j}$ 
the matrix with the corresponding columns of $A$ extracted. Compute 
the sparse QR factorization
\begin{equation}\label{spaqr}
A_j=Q_jR_j
\end{equation}
of $A_j$. (Recall that the sparse QR-factorization is also needed in sparse approximate
inverse computations.)
Assume there can appear $l_j\ll n$ nonzero
entries in $v_j$ at prescribed positions
and denote by $M_j\in \setC^{k_j\times l_j}$ 
the matrix with the corresponding columns of $Q_j^*$ extracted.
Then $v_j$, regarded as a vector in $\setC^{l_j}$, 
of unit norm is computed satisfying
\begin{equation}\label{nocond}
\left|\left|M_jv_j\right|\right| =\left|\left|M_j\right|\right|,
\end{equation}
i.e., $v_j$ is chosen in such a way that 
its component in  the column space of $A_j$ is as large as possible. 
This can be found by computing
the singular value decomposition 
of $M_j$. (Its computational 
cost is completely marginal by the fact that $M_j$ is only
a $k_j$-by-$l_j$ matrix.)

Suppose the column $v_j$ 
has been computed as just described for $j=1,\ldots, n$. 
Then solve the least squares problems
\begin{equation}\label{condwj}
\min_{w_j \in \setC^{k_j}}
\left|\left|A_jw_j-v_j\right|\right|_2
\end{equation}
to have the column $w_j$ of $W$.

For each pair $v_j$ and $w_j$ of columns, the computational cost consists of 
computing the sparse QR factorization \eqref{spaqr} 
and, by using it, solving \eqref{nocond} and \eqref{condwj}.
For the sparse QR factorization there are codes available \cite{DAT}.
(Now $A_j$ has the special property of
being very ``tall and skinny''.)

The constraint of requiring 
the columns of $V$ to be unit vectors is 
actually not a genuine constraint.
That is, the method is scaling invariant from the right and thereby 
any nonzero constraints are acceptable in the sense that the condition
\eqref{nocond} could equally well be replaced with 
$\left|\left|M_jv_j\right|\right| =r_j\left|\left|M_j\right|\right|.$
Let us formulate this as follows.

\begin{theorem} Assume $A\in \setC^{n\times n}$
is nonsingular. If $\mathcal{V}$ and 
$\mathcal{W}$ are standard matrix subspaces of $\setC^{n \times n}$,
then the  factorization \eqref{afac} computed as described is independent 
of the fixed column constraints
$\left|\left|v_j\right|\right|_2=r_j >0$
for $j=1,\ldots, n$.
\end{theorem}

\begin{proof} Let $W$ and $V$ be the matrices
computed with the unit norm constraint for the columns of $V$. Let 
$\hat{W}$ and $\hat{V}$ be computed 
with other strict positivity constraints for the columns of $\hat{V}$,
i.e., \eqref{nocond} is replaced with the condition
\begin{equation}\label{repla}
\left|\left|M_jv_j\right|\right| =r_j\left|\left|M_j\right|\right|.
\end{equation}
Then we have $V=\hat{V}D$ and $W=\hat{W}D$ 
for a diagonal matrix $D$ with nonzero entries.
Consequently, $WV^{-1}=\hat{W}\hat{V}^{-1}$ whenever the factors 
are invertible. 
\end{proof}

\begin{corollary}\label{opto} If a matrix $V$ solving  
%
%
\begin{equation}\label{alkuper}
\min_{W\in \mathcal{W},\; V\in \mathcal{V},\, ||V||_F=1} 
\left|\left| AW-V \right|\right|_F.
\end{equation}
is nonsingular, 
then the factorization \eqref{afac}
coincides with the one  computed to satisfy 
\eqref{nocond} and
\eqref{condwj}.
\end{corollary}

\begin{proof} Suppose $W$ and $V$ solve \eqref{alkuper}.
Since $V$ is 
invertible, we have $\left|\left|v_j\right|\right|_2=r_j>0$.
Using these constraints, compute 
$\hat{W}$ and $\hat{V}$  
to satisfy 
\eqref{repla} and
\eqref{condwj}.
This means solving \eqref{alkuper} columnwise
and thereby the corresponding factorizations coincide.
\end{proof}

It is instructive to see how the computation of an approximate
inverse relates with this. (For sparse approximate inverses
and their historical development, see \cite[Section 5]{BE}.)

\smallskip 

\begin{example} In approximate inverse computations, the
matrix subspace $\mathcal{V}$ is as simple as possible, i.e.,
 the set of diagonal matrices. 
Regarding the constraints, 
the columns are constrained to be unit vectors. Therefore one 
can replace $\mathcal{V}$ with the identity matrix, as is customary. 
See also Example 
\ref{exai} below.
\end{example}

\subsection{Second basic algorithm}
\label{sec:svdalg}
Consider the minimization problem \eqref{aivan} 
under the assumption that 
the columns of $W$ are constrained to be unit vectors instead. 
Based on the sparsity structure of $\mathcal{W}$ and 
the corresponding columns of $A$,
the aim now is at first choosing $W$ optimally. Thereafter $V$ is
determined optimally. The resulting scheme yields a direct analogue of
the power method suggested  in \cite{BHU}. However, 
the method proposed here has
at least three advantages. First, being direct, it seems to be more robust since
there is no need to tune parameters used in the power method.
Second, the Hermitian transpose of $A$ is not needed. Third,
the computational cost is readily predictable by the fact
that, in essence, we only need to compute sparse QR factorizations.

To describe the method, denote by $w_j$ and $v_j$ the $j$th
columns of $W$ and $V$. The column $w_j$ is computed first as follows.
Assume there can appear $k_j\ll n$ nonzero
entries in $w_j$ at prescribed positions
and denote by $A_j\in \setC^{n\times k_j}$ 
the matrix with the corresponding columns of $A$ extracted. 
Assume there can appear $l_j\ll n$ nonzero
nonzero entries in $v_j$ at prescribed positions  and denote by
$\hat{A}_j\in \setC^{(n-l_j)\times k_j}$
the matrix with the corresponding rows of $A_j$ 
removed. Then take $w_j$ to be a right singular vector corresponding
to the smallest singular value of $\hat{A}_j$. 

To have $w_j$ inexpensively, compute 
the sparse QR factorization
$$\hat{A}_j=Q_jR_j$$
of $\hat{A}_j$. Then compute the singular value decomposition of $R_j$.
Of course, its computational cost is completely negligible.
(However, do not form the arising product to have the 
SVD of $\hat{A}_j$ explicitly.) Then take $w_j$ from the singular value
decomposition of $R_j$. 

Suppose the column $w_j$ has been computed as just described for $j=1,\ldots, n$. Then,
to have the columns of $V$, set
$$V=P_{\mathcal{V}}A[w_1\cdots w_n],$$ 
i.e., nonzero entries are accepted only in the
allowed sparsity structure of $v_j$.

For an analogue of Corollary \ref{opto}, 
assume a matrix corresponding to  the smallest
singular value of the linear map \eqref{singva} is nonsingular.
Since $\mathcal{W}$ is a standard matrix
subspace, the computations can be performed columnwise.
The resulting $W$ can be chosen to coincide, once divided
by $\sqrt{n}$, with this matrix.

\subsection{Some general remarks}\label{remarks}
In approximate inverse preconditioning, it is well-known that it
can make a difference whether one computes a right or left approximate
inverse \cite[pp. 449--450]{BE}. 
As we have generalized this technique, this is the case with the
approximate factoring of the inverse also. Here we have considered
only preconditioning from the right.

The usage of standard matrix subspaces leads to maximal parallelizability.
In view of approximating the inverse,
this means that computations are done locally (columnwise) 
and independently, i.e., 
without any global control. 
To compensate for this, with an eye to improve the conditioning of the factors, it seems advisable to impose
additional constraints. This is considered in Section \ref{sec3}.

The simultaneous (somehow optimal) choice of the matrix 
subspaces $\mathcal{W}$ and $\mathcal{V}$ is a delicate matter.
In \cite{BHU} we gave a rule thumb according to which the  
sparsity structures of the matrix
subspaces 
should differ 
as much as possible in approximate factoring of the inverse. 
(This automatically holds in computing  approximate inverses and 
ILU factorizations.)
Numerical 
experiments seem to support this. 
Although we do not quite understand the reasons for this, 
it is partially related with the fact that then there are very few redundancies 
in the factorizations \eqref{afac} as follows.

\begin{proposition} Let $\mathcal{V}$ and $\mathcal{W}$ be standard nonsingular
matrix subspaces of $\setC^{n \times n}$ containing the identity.
If in the complement of the diagonal matrices the intersection of 
$\mathcal{V}$ and $\mathcal{W}$ is empty, then
the maximum rank of the map
\begin{equation}\label{tulo}
(V,W)\longmapsto WV^{-1}
\end{equation}
on $\mathcal{V}\times \mathcal{W}\cap {\rm GL}(n,\setC)$
is $\dim \mathcal{V}+\dim \mathcal{W}-n.$ 
\end{proposition}

\begin{proof}
Linearize the map \eqref{tulo} at $(\hat{V},\hat{W})$ for both 
$\hat{V}$ and $\hat{W}$ invertible. Using the Neumann series
yields the linear term $$\hat{W}(\hat{W}^{-1}W-\hat{V}^{-1}V)\hat{V}^{-1}.$$
At $(\hat{V},\hat{W})=(I,I)$ the rank is $\dim \mathcal{V}+\dim \mathcal{W}-n.$ 
It is the maximum by the fact that
for any nonsingular diagonal matrix $D$ we have $(VD,WD)\longmapsto WV^{-1}$,
i.e., the map \eqref{tulo} can be
regarded as a function of $\dim \mathcal{V}+\dim \mathcal{W}-n$
variables.
\end{proof}

Aside from this basic principle, more 
refined techniques are devised for simultaneously choosing the matrix 
subspaces $\mathcal{W}$ and $\mathcal{V}$ in the sections that follow.
Most notably, optimal ways of choosing $\mathcal{V}$ are devised.

\section{Optimal construction of the matrix subspace $\mathcal{V}$
and imposing constraints}\label{sec3} 
For the basic algorithms introduced, 
a method for optimally choosing  the
matrix subspace $\mathcal{V}$ is devised
under the assumption that the matrix 
subspace  $\mathcal{W}$ has been given. 
Moreover, mechanisms are introduced 
into the basic algorithms that allow
stabilizing the scheme for better conditioned factors. 
(In approximate inverse preconditioning the latter task 
is accomplished in the simplest
possible way: the subspace $\mathcal{V}$ is simply $\setC I$, i.e.,
scalar multiples of the identity.)

\subsection{Optimally constructing the matrix subspace $\mathcal{V}$}
Suppose the matrix subspace  $\mathcal{W}$ has been given.
Then the condition \eqref{nocond} yields a columnwise 
criterion for optimally choosing the sparsity structure 
of the matrix subspace $\mathcal{V}$. 
(Recall that it must be assumed that the nonsingular elements of $\mathcal{V}$ allow
a rapid application of the inverse.)
Once done, proceed by using one of the basic algorithms to compute the factors. 

Consider \eqref{nocond}. 
It is beneficial to choose the sparsity structure of $v_{j}$ in such a way that
the norm of $M_j$ is as large as possible, with the constraint that in the
resulting $\mathcal{V}$ the nonsingular elements are readily invertible. 
In other words,
among admissible columns  of $Q_j^*$, take $l_j$ columns  which 
yields $M_j$ with the maximal norm.
This means that for the optimization problem \eqref{aivan},
with a fixed matrix subspace  $\mathcal{W}$, the matrix subspace
$\mathcal{V}$ is constructed in an optimal way. 


Certainly, the problem of choosing $l_j$ columns to maximize the  norm is combinatorial and thereby rapidly finding a solution does not appear to be straightforward.
A suboptimal choice for the matrix $M_j$ can be readily generated by taking $l_j$
admissible columns of $Q_j^*$ with largest norms. When done with respect to the Euclidean norm,  
the Frobenius norm of the submatrix is maximized instead.
This can be argued, of course, by the fact that 
$$\frac{1}{\sqrt{\max \{k_j,l_l\}}}\left|\left| M_j \right| \right|_F\leq 
\left|\left| M_j \right| \right|\leq \left|\left| M_j \right| \right|_F$$
holds.

This approach starts with $\mathcal{W}$ and then yields $\mathcal{V}$
(sub)optimally.
This process can be used to assess how $\mathcal{W}$ was initially 
chosen. Let us illustrate this with the following example.

\smallskip

\begin{example}
The choice of upper (lower) triangular matrices 
for $\mathcal{V}$
has
the advantage that then we have a warning signal in case
$\mathcal{W}$ is poorly chosen. Namely, 
suppose $\mathcal{V}$ has been (sub)optimally constructed as just described. 
If the factor $V$ computed to satisfy \eqref{aivan} 
is poorly conditioned, 
one should consider
updating the sparsity structure of $\mathcal{W}$ to have a matrix subspace
which better suited for approximate factoring of the inverse
of $A$.\footnote{This is actually the case in the (numerically) exact factoring: To recover whether a 
matrix $A\in \setC^{n \times n}$ is nonsingular, it is
advisable to compute its partially pivoted LU factorization, i.e.,
use a numerically reliable algorithm.}
\end{example}

\smallskip

In this optimization scheme, let us illustrate how the 
matrix subspace $\mathcal{W}$ actually could be poorly chosen.
Namely, the way the above optimization scheme is 
set up means that the sparsity structure
of $\mathcal{W}$ should be such that no two columns share the same 
sparsity structure. (Otherwise $V$ will have equaling columns.)
Of course, this may be too restrictive. In the section that
follows, a way to circumvent this problem is devised by stabilization.

\subsection{Optimizing under additional constraints}
There are instances which require imposing additional constraints
in computing the factors.
Aside from the problems described above,
in tough problems the approximate factors may be poorly 
conditioned of even 
singular.\footnote{This
is a well-known phenomenon in preconditioning. For ILU factorization
there are many ways to stabilize the computations \cite{BE}.
Stabilization has turned out to be indispensable in practice.} 
Because there holds
\begin{equation}
\label{eqn:minbound}
  \frac{\left|\left|AWV^{-1}-I \right| \right| }{
\left|\left| V^{-1}\right| \right|} 
\leq \left|\left| AW-V\right| \right| 
\leq \left|\left| AWV^{-1}-I\right| \right| \left|\left| V\right| \right|,
\end{equation}
this certainly cannot be overlooked.
To overcome this, it is advisable to stabilize the computations by appropriately 
modifying the optimality conditions in computing the factors. 

For the first basic algorithm this means a refined computation of $V$.
Thereafter the factor $W$ is computed columnwise 
as before to satisfy the conditions \eqref{condwj}.
For a case in which the conditioning is readily
controlled,  consider a matrix subspace $\mathcal{V}$ belonging
to the set of  upper (or lower) triangular matrices.
Then, suppose the $j$th column $v_j$  computed
to satisfy \eqref{nocond} results in a tiny $j$th component.
To stabilize the computations for the first basic algorithm, 
we replace $v_j$ by first imposing the $j$th component of $v_j$ to equal 
a constant $r_j>0$. For the remaining components,
let $\hat{M}_j$ be a submatrix consisting of 
the $l_j-1$ largest columns of $Q_j^*$ 
among its first $j-1$ columns.   
Denote the $j$th column of $Q_j^*$ by $p_j$.
Then consider the optimization problem
\begin{equation}\label{misa}
\max_{||\hat{v}_j||_2=1}\left| \left| r_jp_j+\hat{M}_j\hat{v}_j\right| \right|_2.
\end{equation}
By invoking the singular value decomposition $\hat{M}_j=
\hat{U}_j\hat{\Sigma}_j\hat{V}_j^*$ of $\hat{M}_j$,  this is equivalent to solving
\begin{equation}\label{misas}
\max_{||\hat{v}_j||_2=1}\left| \left| r_j\tilde{p}_j+\hat{\Sigma}_j
\tilde{v}_j\right| \right|_2,
\end{equation}
where $\tilde{p}_j=\hat{U}_j^*p_j$ and $\tilde{v}_j=\hat{V}_j^*\hat{v}_j$.
Consequently, choose $\tilde{v}_j=(e^{i \theta},0,0,\ldots,0)$, where
$\theta$ is the argument of the first component of 
$\tilde{p}_j$. (If the first component is zero, then any $\theta$ will
do.)
Set the column $v_j$ to be the sum of $r_je_j$ and the vector
obtained after putting the entries of 
$\hat{V}_j\tilde{v}_j$ 
at the positions where the corresponding
$l_j-1$ largest columns of $Q_j^*$ appeared.
(Here $e_j$ denotes
the $j$th standard basis vector of $\setC^n$.)

Observe that the solution does not depend on the value of 
$r_j>0$. In particular,
it is not clear how large $r_j$ should be.

Again it is instructive to contrast this with the approximate
inverse computations.

\smallskip

\begin{example}\label{exai}
The sparse approximate inverse computations yield the simplest
case of imposing additional constraints as just described.
That is, the sparse approximate inverse computations can be interpreted as
having $l_j=1$ for every column,
combined with imposing $r_j=1$.
\end{example}

\smallskip

The LU factorization and thereby triangular matrices
are extensively used in preconditioning. Because the LU factorization
without pivoting is unstable, some kind of stabilization is needed.
It is clear that the QR factorization also gives reasons to look at 
triangular matrices. The approach differs from that of using the LU factorization in
that its computation does not require a stabilization, i.e.,
nothing like partial pivoting is needed. Of course,
our intention is not to propose computing the full QR factorization.
Understanding the Q factor is critical as follows.
\smallskip

\begin{example}\label{spqr}
The QR factorization $A^*=QR$ of the Hermitian transpose of 
$A$ can be used as a starting point to construct matrix subspaces 
for approximate
factoring of the inverse. 
Namely, we have $AQ=R^*$. Therefore 
$\mathcal{V}$ belonging to the set of lower triangular matrices
is a natural choice. For $\mathcal{W}$ one needs to generate an 
approximation to the sparsity structure of $Q$. For this there are many
alternatives.  
\end{example}

\smallskip

Aside from upper (lower) triangular 
matrices, the are, of course, completely different alternatives.
Consider, for example,  choosing $V$ among diagonally dominant matrices.
Since the set of diagonally dominant matrices
is not a matrix subspace, dealing with this structure requires
using constraints. 
It is easy to see that the problem
can be tackled  completely analogously, 
by imposing 
 imposing $r>1$ to hold for every diagonal entry. Thereafter \eqref{misa} 
solved for having the other components in the column. 
The inversion of $V$ can be performed by simple
algorithms such as the Gauss-Seidel method.

\section{Constructing the matrix 
subspace $\mathcal{W}$}\label{Sec4}  
Optimally constructing the matrix subspace $\mathcal{W}$ 
for approximate factoring of the inverse appears seemingly challenging. 
Some ideas are suggested in what
follows, although no claims concerning the optimality are made. 
We suggest starting
the process by taking an initial standard matrix subspace $\mathcal{V}_0$
which precedes the actual $\mathcal{V}$. 
Once $\mathcal{W}$ has been as constructed, then
$\mathcal{V}_0$ should be replaced
with $\mathcal{V}$ computed
with the techniques introduced in Section \ref{sec3}. 

\subsection{The Neumann series constructions}
\label{sec:wcalheur}
For  approximate inverse computations
the selection of an
a-priori sparsity pattern 
is a  well-known problem
\cite{Chow2000, Benzi1997}.
Good sparsity patterns are, at least in some cases, related to the
transitive closures 
of subsets of the connectivity graph of 
$G(A)$ of $A$. 
This can also be interpreted as computing
level set expansions on the vertices of a sparsified $G(A)$.

In \cite{Chow2000} numerical dropping is used
to sparsify $G(A)$ or its level set
expansions. 
Denote by $v\in \Cset^n$ a
vector with entries $v_j$. To select the relatively large entries of
$v$ numerically, entries are dropped by relative tolerance $\tau$ and
by count $p$, i.e., only those entries of $v$ that are relatively
large with the restriction of $p$ largest entries at most are
stored. 
(Note that the diagonal elements are
not subjected to numerical dropping.)
In what follows, these rules are referred to as numerical
dropping by tolerance and count. 

The dropping can be performed on an initial matrix or during the
intermediate phases of the level set expansion. Thus we have two sets
of parameters $(\tau_i,p_i)$ controlling the initial sparsification and
$(\tau_l,p_l)$ controlling the sparsification during level set
expansion. In addition, we adopt the convention that setting any
parameter as zero implies that the dropping parameter is not used.

With these preparations for approximate factoring of the inverse, 
take an initial standard matrix subspace $\mathcal{V}_0$ and 
consider generating a sparsity pattern for 
$\Wcal$. 
Assuming $V_0=P_{\Vcal_0}A\in \Vcal_0$ is invertible, we
have
\begin{equation*}
  A=V_0(I-V_0^{-1}(I-P_{\Vcal_0})A)=V_0(I-S).
\end{equation*}
Whenever $\knorm{S}{}<1$, there holds
$A^{-1}=(I+\sum_{j=1}^\infty S^j)V_0^{-1}=WV_0^{-1}$
by invoking the Neumann series. Therefore then
\begin{equation}
  \label{eqn:factspar}
W=I+\sum_{j=1}^\infty S^j.
\end{equation} 
Although the assumption $\knorm{S}{}<1$ is generally too strict in practice, 
we may formally truncate 
the series \eqref{eqn:factspar} 
to generate a sparsity pattern.
To make this economical and to retain $\mathcal{W}$ sparse enough, 
compute powers of $S$ only approximately
by using sparse-sparse operations combined with numerical dropping
and level of fill techniques.

Observe that, to operate with the series \eqref{eqn:factspar} we need 
$S=V_0^{-1}(I-P_{\Vcal_0})A$.  It is this which requires setting an initial 
standard matrix subspace $\Vcal_0$. 

\smallskip

\begin{example}
For $S=V_0^{-1}(I-P_{\Vcal_0})A$ we need to set an initial standard matrix subspace. 
The most inexpensive alternative is to take 
$\mathcal{V}_0$ to be the set of diagonal matrices. 
Then $V_0=P_{\Vcal_0}A$ is an immediately found.
\end{example}

\smallskip

There are certainly other inexpensive alternatives for $\Vcal_0$, such
as block diagonal matrices. Once fixed, thereafter the scheme
can be given as Algorithm \ref{alg:factspar} below.

\begin{algorithm}[t]
   \caption{Sparsified powers for constructing $\Wcal$}
   \label{alg:factspar} 
   \begin{algorithmic}[1]
     \State Set a truncation parameter $k$
     \State Compute $V_0^{-1}$
     \State Compute $S=V_0^{-1}(I-P_{\Vcal_0})A$ 
     \State Apply numerical dropping by tolerance and count to columns of $S$
     \For{{\bf columns} $j$ {\bf in parallel}}
     \State Set $s_j=t_j=e_j$
     \For{$l=1,\ldots,k$}
     \State Compute $t_j=St_j$
     \State Apply numerical dropping by tolerance and count to $t_j$
     \State Compute $s_j=s_j+t_j$
     \EndFor
     \State Set sparsity structure of $w_j$ to be the sparsity structure of $s_j$
     \EndFor
     \State Set $\Wcal=\Wcal\setminus\{\Vcal_0\setminus \mathcal{I}\}$ 
   \end{algorithmic}
\end{algorithm}


Note that final step of Algorithm \ref{alg:factspar} is to keep the
intersection of $\Wcal$ and $\Vcal_0$ empty apart from the
diagonal; see Section \ref{remarks}.
After the sparsity structure for a matrix subspace $\Wcal$
has been generated, the sparsity structure of $\Vcal_0$ can be updated to be $\mathcal{V}$
by using $\Wcal$.



\subsection{Algebraic constructions}
Next we consider some  purely algebraic arguments  
which might be of use in constructing $\mathcal{W}$. Again start with
an initial standard matrix subspace $\mathcal{V}_0$.
Take the sparsity structure of the $j$th column of 
$\mathcal{V}_0$ and consider the corresponding 
rows of $A\in \setC^{n \times n}$. 
Choose the sparsity structure of
the $j$th column of $\mathcal{W}$ to be the union
of the sparsity structures of these rows. 
This is a necessary (but not sufficient) condition
for $A\mathcal{W}$ to have an intersection with $\mathcal{V}_0$. 
This simply means 
choosing $\mathcal{W}$ to have the 
sparsity structure of $A^*\mathcal{V}_0$.

Most notably, the process is very inexpensive and can be executed in parallel.
One only needs to control that the columns of $\mathcal{W}$ remain
sufficiently sparse. With probability one, the following 
algorithm yields the desired sparsity structure.




\begin{algorithm}[H]
\caption{Computing a sparsity structure for $\mathcal{W}$}
\label{algor}
\begin{algorithmic}[2]
\Require A sparse matrix $A \in \setC^{n \times n}$ and
a random column $v_j \in \mathcal{V}_0$.
\Ensure Sparsity structure of the column $w_j$.
\State Compute $w=A^*v_j$
\If{$w$ is not sparse enough}
\State Sparsify $w$ to have the sparsity structure of $w_j$.
\EndIf
\State Take the sparsity structure of  $w_j$ to be the sparsity structure of $w$.
\end{algorithmic}
\end{algorithm}

Observe that we do not have $A^*\mathcal{V}_0=\mathcal{W}$ since
the computation is concerned with sparsity structures.

Approximate inverse
preconditioning corresponds to choosing $\mathcal{V}_0$ 
to be the set of diagonal matrices. 
Then the sparsity structure of 
$\mathcal{W}$ equals that of $A^*$.
The following two
examples illustrate two extremes cases of
this choice.

\smallskip

\begin{example}\label{teje} 
Take $\mathcal{V}_0$ to be the set of diagonal
matrices. Then  the first basic algorithm reduces to
the approximate inverse preconditioning. 
Algorithm \ref{algor} yields now a standard matrix subspace $\mathcal{W}$  
whose sparsity structure equals that of $A^*$. 
This can yield very good results. 
If $A$ has orthogonal rows (equivalently, columns)
then and only then this gives exactly a correct matrix 
subspace $\mathcal{W}$ for factoring the inverse of $A$
as $AWV^{-1}=I$ 
 when $\mathcal{V}$ is taken
to be $\mathcal{V}_0$.\footnote{In view of this, it
seems like a natural problem
to ask, how well $A$ can be approximated with
matrices of the form $DU$ with $D$ diagonal and $U$ unitary.}
\end{example}

\smallskip

Having identified an ideal structure for
the approximate inverse preconditioning when $\mathcal{W}$ 
is constructed with Algorithm \ref{algor}, 
how about when $A$ is far from being a scaled
unitary matrix? An upper (lower) triangular matrix
is a scaled unitary matrix only when it reduces to a 
diagonal matrix.

\smallskip

\begin{example}\label{teje2} 
Take again $\mathcal{V}_0$ to be the set of diagonal
matrices. Then the basic algorithm reduces to
the approximate inverse preconditioning.
Algorithm \ref{algor} yields a standard matrix subspace $\mathcal{W}$  
whose sparsity structure equals that of $A^*$. 
This yields very poor results if  
$A$ is an upper (lower) triangular matrix. Namely, then its inverse is
also upper (lower) triangular. 
\end{example}

\smallskip

Algorithm \ref{algor} is
set up in such a way that
if $\mathcal{V}_0\subset \tilde{\mathcal{V}_0}$, then
$\mathcal{W}\subset \tilde{\mathcal{W}}$.
Thereby  matrix subspaces can be constructed 
to handle the two extremes of 
Examples \ref{teje} and \ref{teje2}
simultaneously.

In practice $\mathcal{V}_0$ should be 
more complex, i.e., the set of diagonal matrices
is a too simple structure. 
One option is to start with $\mathcal{V}_0$ having 
the sparsity structure of the Gauss-Seidel 
preconditioner. 

\begin{definition} A standard matrix subspace $\mathcal{V}$ of $\setC^{n \times n}$ is said to have 
the sparsity structure of the Gauss-Seidel 
preconditioner of $A\in \setC^{n \times n}$ if the nonzero
entries in $\mathcal{V}$ 
appear on the diagonal and 
there where the strictly lower (upper) triangular part of
$A$ has nonzero entries.
\end{definition}

\section{Numerical experiments}

The purpose of this final section is to illustrate, with the help of four
numerical experiments, how the preconditioners devised in Sections
\ref{sec:DIAF} and \ref{sec3} perform in practice. Since there is an
abundance of degrees of freedom to construct matrix subspaces for
approximate factoring of the inverse, only a very incomplete set of
experiments can be presented. In particular, we feel that there is a
lot of room for new ideas and improvements.

In choosing the benchmark sparse linear systems, we used the
University of Florida collection \cite{Davis2011}.  
The problems were selected to be the most challenging ones
to precondition among those tested in \cite{Benzi2000}.  
For the matrices used and some of their properties, see Table
\ref{tab:testprob_small}.
Assuming the
reader has an access to \cite{Benzi2000}, the comparison between the
methods proposed here and the diagonal Jacobi preconditioning,
\ilut$(0)$, \ilut$(1)$, \ilut\ and \ainv\ can be readily made.
For a comparision between ILUs and AINV,
see, e.g., \cite{BOS}.

\begin{table}[t]
\begin{center}
  \begin{tabular}{|*{5}{c|}}
    \hline
    Problem    & Area & $n$ & $\nz{A}$ & $k_1=\nz{A}/n$ \\ 
    \hline
    west1505 & Chemical engineering & 1505    & 5414     & 3.6 \\ 
    west2021 & Chemical engineering & 2021    & 7310     & 3.62 \\ 
    lhr02    & Chemical engineering & 2954    & 36875    & 12.5 \\ 
    bayer10  & Chemical engineering & 13436   & 71594    & 5.33 \\ 
    sherman2 & PDE                  & 1080    & 23094    & 21.4 \\ 
    gemat11  & Linear programming   & 4929    & 33108    & 6.72 \\ 
    gemat12  & Linear programming   & 4929    & 33044    & 6.7 \\ 
    utm5940  & PDE                  & 5940    & 83842    & 14.1 \\ 
    e20r1000 & PDE                  & 4241    & 131430 & 31 \\ 
    \hline
  \end{tabular}
  \caption{Matrices of the experiments, their application area, size,
    number of nonzeros and density.}
  \label{tab:testprob_small}
\end{center}
\end{table}

Regarding preprocessing, in each experiment the original matrix has
been initially permuted to have nonzero diagonal entries and scaled
with \mcSF. (See \cite{Duff2001} for \mcSF.) It is desirable that the
matrix subspace $\Vcal$ contains hierarchically connected parts of the
graph of the matrix. To this end we use an approach to find the
strongly connected subgraphs of the matrix; see Duff and Kaya
\cite{Duff2011}. We then obtain a permutation $P$ such that after the
permutations, the resulting linear system can be split as
\begin{equation}\label{splitti}
Ax=(L+D+U)x=b, 
\end{equation}
where $L^T$ and $U$ are strictly block upper triangular and $D$ is a
block diagonal matrix. The construction of this permutations consumes
at most $\Oh{n\log{(n)}}$ operations.\footnote{Preprocessing is
  actually a part of the process of constructing the matrix subspaces
  $\mathcal{W}$ and $\mathcal{V}$. That is, it is insignificant
  whether one orders correspondingly the entries of the matrix or the
  matrix subspaces.}

In the experiments, the right-hand side $b \in \setC^n$ in
\eqref{splitti} was chosen in such a way that the solution of the
original linear system was always $x=(1,1,\ldots,1)$.  As in
\cite{Benzi2000}, as a linear solver we used
\bicgstab\ \cite{vanVorst1992}. The iteration was considered converged
when the initial residual had been reduced by eight orders of
magnitude.

The numerical experiments were carried out with
\matlab \footnote{VersionR2010a.}.

\medskip

\begin{example}
\label{ex:paifdiaf1}

We compare the minimization algorithm presented \cite{BHU} (\paif)
with the QR factorization based minimization algorithm of Section
\ref{sec:qralg} (\diaf-Q). We construct $\Wcal$ with the
heuristic Algorithm \ref{alg:factspar} of Section
\ref{sec:wcalheur}. For all test matrices, we use $k=3$ and
$\tau_i=1E-1$, $p_i=0$, $\tau_l=0$ and $p_l=0$, as parameters. For
\paif, $80$ refinement iterations were always used which is a somewhat
more than what we have found to be necessary in practice. However, we
want to be sure that the comparison is descriptive in terms of the
quality of the preconditioner.

We choose $\Vcal$ to be the subspace of block diagonal matrices with
block bounds and sparsity structure chosen according to the block
diagonal part of $A$, i.e., the matrix $D$ in \eqref{splitti}.  Then
in the heuristic construction of $\Wcal$ with Algorithm
\ref{alg:factspar},  $V_0$ is taken to be a diagonal matrix.

We denote by $|D_j|_M$ the maximum blocksize of $\Vcal$ and by \#$D_j$
the number of blocks in $\Vcal$ in total. Density of the
preconditioner, denoted by $\rho$, is computed as
$\rho=(\nz{W}+\nz{L_{V}}+ \nz{U_{V}})/\nz{A}$, where $\nz{A}$,
$\nz{W}$, $\nz{L_{V}}$ and $\nz{U_{V}}$ denote the number of nonzeroes
in $A$, $W$ and the LU decomposition of $V$. For both \paif\ and
\diaf-Q, we also compute the condition number estimate $\kappa(V)$ and
norm of the minimizer $\knorm{AW-V}{F}$, denoted by nrm. Finally, its
denotes the number of \bicgstab\ iterations. By $\dagger$ we denote if
no convergence of \bicgstab\ within $1000$ iterations. Breakdown of
\bicgstab\ is denoted by $\ddag$. Table \ref{tab:paifdiaf1_res} shows
the results.
\begin{table}[H]
  \begin{center}
    \begin{footnotesize}
  \begin{tabular}{|*{10}{c|}}
    \hline
 \multicolumn{4}{|c|}{}& \multicolumn{3}{l|}{\paif}& \multicolumn{3}{l|}{\diaf-Q} \\
 Problem  &  $|D_j|_M$  &  \#$D_j$  &  $\rho$  &  $\kappa(V)$  & nrm  &  its  &  $\kappa(V)$  &  nrm  &  its  \\
 \hline
 west1505 & 50 & 34 & 2.75 & 3.17E+04 & 3.59 & 18 & 1.85E+03 & 3.49 & 18 \\
 west2021 & 50 & 47 & 2.69 & 5.53E+03 & 3.84 & 23 & 3.33E+03 & 3.53 & 26 \\
 lhr02 & 50 & 66 & 1.11 & 1.65E+03 & 6.69 & 24 & 9.05E+02 & 7.01 & 32 \\
 bayer10 & 250 & 67 & 2.56 & 8.00E+05 & 22.27 & 56 & 2.50E+05 & 14.27 & 36 \\
 sherman2 & 50 & 24 & 1.05 & 4.32E+02 & 2.45 & 5 & 3.77E+02 & 1.84 & 5 \\
 gemat11 & 50 & 115 & 1.91 & 2.28E+05 & 3.58 & 109 & 1.50E+05 & 2.79 & 68 \\
 gemat12 & 50 & 114 & 1.91 & 5.96E+06 & 6.87 & 77 & 4.58E+06 & 5.20 & 77 \\
 utm5940 & 250 & 29 & 1.73 & 3.91E+06 & 14.66 & 295 & 1.84E+06 & 12.86 & 221 \\
 e20r1000 & 200 & 27 & 4.23 & 3.22E+06 & 13.67 & 465 & 4.44E+03 & 8.82 & 364 \\
    \hline
  \end{tabular}
  \end{footnotesize}
  \caption{Comparison of \paif\ and \diaf-Q\ algorithms}
  \label{tab:paifdiaf1_res}
  \end{center}
\end{table}

Results very similar to those seen in Table \ref{tab:paifdiaf1_res}
were also observed in other numerical tests that were conducted. As a
general remark, the iteration counts with \bicgstab\ when
preconditioned with \diaf-Q\ are not dramatically different from those
achieved with \paif.  The main benefits of \diaf-Q\ are that neither
the Hermitian transpose of $A$ is required in the computations nor an
estimate for the norm of $A$. Moreover, \diaf-Q\ is a direct method,
so that its computational cost is easily estimated, while it is not so
clear when to stop the iterations with \paif.

The computational cost and parallel implementation of \diaf-Q\ is very
similar to the established preconditioning techniques based on norm
minimization for sparse approximate inverse. 
(For these issues, see \cite{Chow2000}.)  
That is, \diaf-Q\ scales essentially accordingly in terms
of the computational cost and parallelizability properties.
\end{example}

\begin{example}
\label{ex:paifdiaf2}

Next we compare \paif\ with the SVD based algorithm of Section
\ref{sec:svdalg} (\diaf-S). Again $\Wcal$ is constructed
with the heuristic  Algorithm \ref{alg:factspar} 
of Section
\ref{sec:wcalheur}.  All the parameters were kept the same as in the
previous example, i.e., $k=3$ and $\tau_i=1E-1$, $p_i=0$, $\tau_l=0$
and $p_l=0$. Also, $80$ refinement steps were again used in the power
method, so that the results for \paif\ are identical to those
presented in Example \ref{ex:paifdiaf1}.

Table \ref{tab:paifdiaf2_res} shows the results.
\begin{table}[H]
  \begin{center}
    \begin{footnotesize}
  \begin{tabular}{|*{10}{c|}}
    \hline
 \multicolumn{4}{|c|}{}& \multicolumn{3}{l|}{\paif}& \multicolumn{3}{l|}{\diaf-S} \\
 Problem  &  $|D_j|_M$  &  \#$D_j$  &  $\rho$  &  $\kappa(V)$  & nrm  &  its  &  $\kappa(V)$  &  nrm  &  its  \\
 \hline 
 west1505 & 50 & 34 & 2.75 & 3.17E+04 & 3.59 & 18 & 4.06E+03 & 3.04 & 14 \\
 west2021 & 50 & 47 & 2.69 & 5.53E+03 & 3.84 & 23 & 8.31E+03 & 3.26 & 27 \\
 lhr02 & 50 & 66 & 1.11 & 1.65E+03 & 6.69 & 24 & 1.90E+03 & 5.68 & 55 \\
 bayer10 & 250 & 67 & 2.56 & 8.00E+05 & 22.27 & 56 & 9.50E+05 & 11.68 & 46 \\
 sherman2 & 50 & 24 & 1.05 & 4.32E+02 & 2.45 & 5 & 4.33E+02 & 1.67 & 5 \\
 gemat11 & 50 & 115 & 1.91 & 2.28E+05 & 3.58 & 109 & 2.28E+05 & 2.90 & 113 \\
 gemat12 & 50 & 114 & 1.91 & 5.96E+06 & 6.87 & 77 & 1.51E+08 & 3.72 & 201 \\
 utm5940 & 250 & 29 & 1.73 & 3.91E+06 & 14.66 & 295 & 3.91E+06 & 7.43 & $\ddag$ \\
 e20r1000 & 200 & 27 & 4.23 & 3.22E+06 & 13.67 & 465 & 1.96E+04 & 10.37 & 444 \\
    \hline
  \end{tabular}
  \end{footnotesize}
  \caption{Comparison of \paif\ and \diaf-S\ algorithms}
  \label{tab:paifdiaf2_res}
  \end{center}
\end{table}

The results of Table \ref{tab:paifdiaf2_res} with \diaf-S are very
similar to those in Table \ref{tab:paifdiaf1_res}. The only notable
exception is the matrix utm5940, for which no convergence was achieved
with \diaf-S. With the metrics used, we do not quite understand why
\diaf-S fails to produce a good preconditioner for this particular
problem. The computed norm $\knorm{AW-V}{F}$ is smaller than the one
attained with \diaf-Q and the condition number estimate is only
slightly worse. The reason is most likely related with the fact that
the bound \eqref{eqn:minbound} cannot be expected to be
tight enough when $\kappa(V)$ is large. 

\end{example}

The following example illustrates how the matrix subspace $\Vcal$ can
be optimally constructed with the techniques of Section \ref{sec3}.

\begin{example}
  \label{ex:optconst_diag}

  In this example we consider an optimal construction of $\Vcal$. To this
  end, we first construct $\Wcal$ with the heuristic Algorithm
  \ref{alg:factspar} presented in Section \ref{sec:wcalheur}. Then, to
  construct $\Vcal$, we apply the techniques presented in Section
  \ref{sec3}. After the sparsity structures  of the subspaces have been
  fixed, the resulting minimization problem is solved with
  \diaf-Q.

  Consider the minimization problem \eqref{condwj}. If no restrictions
  on the number of nonzero entries in a matrix subspace $\Vcal$
  are imposed, the norm $\knorm{AW-V}{F}$ can be
  decreased by choosing as many entries as possible from the sparsity
  structure of  $AW$ to be in the sparsity 
  structure of  $\Vcal$.\footnote{For example, $\mathcal{V}$ can never 
  be the full set set of upper triangular matrices since it 
  would require storing $O(n^2)$ complex numbers. The problem is then,
  how to choose a subspace $\mathcal{V}$ of upper triangular matrices.}
  To illustrate this, we take 
  $\mathcal{V}$ to be a subspace of block digonal matrices
  by allowing only certain degree of sparsity $k_{\Vcal}$ per column. 
  The nonzero entries are chosen with the techniques of  Section \ref{sec3}. 
 
  We again set $k=3$ and $\tau_i=1E-1$, $p_i=0$, $\tau_l=0$ and
  $p_l=0$, as parameters for all test matrices.
  To have the locations for the entries in the diagonal blocks of
  $\Vcal$, we then apply the method presented in Section
  \ref{sec3}. Subspace $\mathcal{W}$ is constructed with the heuristic 
  Algorithm \ref{alg:factspar} by
  setting $\Vcal_0$ to be a subspace of block diagonal matrices with full
  blocks. 
  This is to ensure that intersection of the final $\Vcal$ and
  $\Wcal$ is empty.

  Table \ref{tab:paifdiaf2_res} shows the results, where at most
  $k_{\Vcal}$ entries in each column of the sparsity pattern of
  $\Vcal$ have been allowed. For each test problem we have used the
  same block structure as in Examples \ref{ex:paifdiaf1} and
  \ref{ex:paifdiaf2}, only the locations of the nonzero entries in
  $\Wcal$ and $\Vcal$ is varied.
  \begin{table}[H]
    \begin{center}
      \begin{scriptsize}
        \begin{tabular}{|*{13}{c|}}
          \hline
          & \multicolumn{4}{|l|}{$k_{\Vcal}=10$}& \multicolumn{4}{l|}{$k_{\Vcal}=30$}& \multicolumn{4}{l|}{$k_{\Vcal}=50$} \\
        Problem  &  $\rho$  &  $\kappa(V)$  & nrm  &  its  & $\rho$  &  $\kappa(V)$  & nrm  &  its & $\rho$  &  $\kappa(V)$  & nrm  &  its   \\
        \hline 
        west1505 & 2.81 & 2.24E+03 & 3.54 & 20 & 2.83 & 4.57E+03 & 3.36 & 20 & 2.83 & 4.57E+03 & 3.36 & 20 \\
        west2021 & 2.75 & 5.11E+04 & 3.70 & 30 & 2.76 & 5.67E+04 & 3.46 & 31 & 2.76 & 5.67E+04 & 3.46 & 31 \\
        lhr02 & 1.15 & 1.07E+04 & 6.95 & 47 & 1.17 & 1.27E+04 & 6.92 & 43 & 1.17 & 1.27E+04 & 6.92 & 43 \\
        bayer10 & 2.68 & 6.72E+07 & 13.83 & 104 & 2.75 & 1.86E+05 & 13.42 & 29 & 2.76 & 1.86E+05 & 13.42 & 31 \\
        sherman2 & 1.01 & 1.49E+02 & 1.88 & 7 & 1.11 & 3.77E+02 & 1.82 & 5 & 1.11 & 3.77E+02 & 1.82 & 5 \\
        gemat11 & 2.14 & 1.50E+05 & 2.84 & 81 & 2.24 & 1.50E+05 & 2.75 & 68 & 2.24 & 1.50E+05 & 2.75 & 69 \\
        gemat12 & 2.09 & 4.24E+06 & 5.21 & 80 & 2.17 & 4.58E+06 & 5.13 & 71 & 2.17 & 4.58E+06 & 5.13 & 70 \\
        utm5940 & 2.11 & 1.82E+06 & 13.12 & $\ddag$ & 2.49 & 2.09E+06 & 12.72 & 209 & 2.51 & 2.11E+06 & 12.71 & 201 \\
        e20r1000 & 3.77 & 2.16E+06 & 20.66 & $\ddag$ & 4.99 & 1.01E+05 & 11.76 & $\ddag$ & 5.49 & 6.67E+04 & 8.77 & 418 \\
        \hline
      \end{tabular}
    \end{scriptsize}
    \caption{Adaptive selection of $\Vcal$ for different values of $k_{\Vcal}$}
    \label{tab:paifdiaf3_res}
  \end{center}
\end{table}
  
  As seen in Table \ref{tab:paifdiaf3_res}, choosing more entries in
  $\Vcal$, i.e., increasing $k_{\Vcal}$ always improves the norm of
  the minimizer, which is well supported by the theory. Allowing more
  entries in $\Vcal$ generally produces a better preconditioner. In a
  few cases where a slightly worse convergence can be observed, we also
  observe a slightly worse condition number estimate for the computed
  $V$.

\end{example}

The final example illustrates the optimal selection of a block upper
triangular subspace $\Vcal$ as well as optimization under additional
constraints. 

\begin{example}
  \label{ex:optconst_ut}

  We consider optimal construction $\Vcal$ in the case where $\Vcal$
  is block upper triangular. As parameters we again use $k=3$ and
  $\tau_i=1E-1$, $p_i=0$, $\tau_l=0$ and $p_l=0$ and use the strongly
  connected subgraph approach to have a block structure for the
  subspace $\Vcal$.

  To have locations for the entries in the block upper triangular
  $\Vcal$, we apply the method presented in Section \ref{sec3}. As in
  Example \ref{ex:optconst_diag}, we construct $\Wcal$ with Algorithm
  \ref{alg:factspar} by setting $\Vcal_0$ as a subspace of block upper
  triangular matrices with full blocks. The resulting $\Wcal$ is a
  lower triangular matrix subspace consisting of a diagonal part and a
  strictly block lower triangular part. The resulting minimization
  problem is solved with \diaf-Q. Note that by the structure of such a
  subspace, the conditioning of $W\in\Wcal$ can be readily verified.

  Table \ref{tab:paifdiaf2_res} shows the results, where at most
  $k_{\Vcal}$ entries in each column in the block upper triangular
  part of $\Vcal_0$ have been allowed. The used block structure is the
  same as in Examples \ref{ex:paifdiaf1}, \ref{ex:paifdiaf2} and
  \ref{ex:optconst_diag}, only the locations and
  the number of the nonzero entries is varied in $\Wcal$ and $\Vcal$.
  \begin{table}[H]
    \begin{center}
      \begin{scriptsize}
        \begin{tabular}{|*{13}{c|}}
          \hline
          & \multicolumn{4}{|l|}{$k_{\Vcal}=10$}& \multicolumn{4}{l|}{$k_{\Vcal}=30$}& \multicolumn{4}{l|}{$k_{\Vcal}=100$} \\
        Problem  &  $\rho$  &  $\kappa(V)$  & nrm  &  its  & $\rho$  &  $\kappa(V)$  & nrm  &  its & $\rho$  &  $\kappa(V)$  & nrm  &  its   \\
        \hline 
        west1505 & 2.37 & 4.45E+06 & 4.79 & 860 & 2.68 & 1.88E+06 & 2.52 & 78 & 2.75 & 6.95E+04 & 2.42 & 24 \\
        west2021 & 2.31 & 1.66E+11 & 5.34 & $\dagger$ & 2.57 & 8.30E+05 & 2.53 & 67 & 2.66 & 3.25E+05 & 2.27 & 23 \\
        lhr02 & 1.03 & 2.45E+03 & 4.22 & 36 & 1.40 & 4.89E+03 & 3.21 & 22 & 1.46 & 4.93E+03 & 3.18 & 21 \\
        bayer10 & 2.09 & 3.74E+09 & 11.76 & 963 & 2.42 & 2.90E+06 & 8.19 & 390 & 2.49 & 3.39E+06 & 8.10 & 16 \\
        sherman2 & 0.89 & 1.61E+02 & 0.89 & 6 & 1.24 & 4.66E+02 & 0.63 & 3 & 1.24 & 4.66E+02 & 0.63 & 3 \\
        gemat11 & 1.84 & 1.50E+05 & 2.37 & 55 & 1.96 & 1.60E+05 & 1.74 & 32 & 1.97 & 1.60E+05 & 1.74 & 32 \\
        gemat12 & 1.82 & 1.45E+09 & 3.37 & 160 & 1.95 & 1.99E+09 & 2.38 & 39 & 1.96 & 2.03E+09 & 2.37 & 38 \\
        utm5940 & 1.60 & 5.29E+07 & 9.16 & 653 & 2.07 & 1.12E+09 & 8.17 & 141 & 2.18 & 8.56E+08 & 8.14 & 165 \\
        e20r1000 & 1.90 & 2.00E+11 & 22.32 & $\dagger$ & 2.77 & 6.59E+10 & 12.48 & $\dagger$ & 3.85 & 2.52E+09 & 6.05 & 792 \\
        \hline
      \end{tabular}
    \end{scriptsize}
    \caption{Adaptive selection of $\Vcal$ for different values of $k_{\Vcal}$}
    \label{tab:paifdiaf4_res}
  \end{center}
\end{table}
  
  The results of Table \ref{tab:paifdiaf4_res} are very similar to
  those of Table \ref{tab:paifdiaf3_res}. Again, allowing more entries
  in $\Vcal$ always improves the norm of the minimizer and usually
  also produces a better preconditioner. 

  Similarly as in Example \ref{ex:paifdiaf2}, it is again hard to
  understand why $k_{\Vcal}=100$ produces a worse preconditioner than
  $k_{\Vcal}=30$ for utm5940. We attribute this behaviour to the
  looseness of the bound \eqref{eqn:minbound}, i.e., when $\kappa(V)$
  is large, the minimization of $\knorm{AW-V}{F}$ may not compensate
  sufficiently  for this in these cases.

  To improve the conditioning of $V\in\Vcal$, we now consider the same
  problems using the technique of imposing constraints as described in
  Section \ref{sec3}. As a constraint we require that for the diagonal
  entries $v_{jj}$ of $V$ it holds $\absnorm{v_{jj}}\geq 1e-2$. In
  case the requirement is not met, we impose a constraint with
  $r=2$. Table \ref{tab:paifdiaf4_resstab} describes the results for
  $k_{\Vcal}=100$, where the number of constrained columns is denoted
  by stab.
  
  \begin{table}[H]
    \begin{center}
      \begin{scriptsize}
        \begin{tabular}{|*{6}{c|}}
          \hline
          & \multicolumn{5}{l|}{$k_{\Vcal}=30$} \\
        Problem  &  $\rho$  &  $\kappa(V)$  & nrm  &  stab & its  \\
        \hline 
        west1505 & 2.75 & 6.16E+04 & 3.06 & 1 & 25 \\
        west2021 & 2.66 & 1.63E+05 & 2.92 & 1 & 21 \\
        lhr02 & 1.46 & 4.93E+03 & 3.18 & 0 & 21 \\
        bayer10 & 2.49 & 3.39E+06 & 8.10 & 0 & 16 \\
        sherman2 & 1.24 & 4.66E+02 & 0.63 & 0 & 3 \\
        gemat11 & 1.97 & 1.60E+05 & 1.74 & 0 & 32 \\
        gemat12 & 1.96 & 2.03E+09 & 2.37 & 0 & 38 \\
        utm5940 & 2.18 & 7.56E+06 & 8.24 & 1 & 113 \\
        e20r1000 & 3.85 & 8.66E+09 & 6.27 & 3 & 738 \\
        \hline
      \end{tabular}
    \end{scriptsize}
    \caption{Constrained selection of $\Vcal$ for $k_{\Vcal}=100$}
    \label{tab:paifdiaf4_resstab}
  \end{center}
\end{table}

  As seen in Table \ref{tab:paifdiaf4_resstab}, if only a small number
  columns has to be constrained, the technique is be effective. In
  other numerical experiments not reported here we observed that if
  too many columns have to be constrained, the norm of the minimizer
  $\knorm{AW-V}{F}$ tends to increase. An approach to find a right balance is
   needed then. 

\end{example}

\medskip

To sum up these experiments, the iteration counts obtained with
\diaf-Q\ and \diaf-S\ (which are fully parallelizable)
seem to be competitive with the iteration counts
obtained with the standard algebraic (sequential) preconditioning
techniques. Moreover, a good problem specific tuning of matrix subspaces 
possesses a lot of potential for significantly speeding up the iterations.

\end{document}